\documentclass[a4paper,11pt,oneside]{amsart}
\pdfoutput=1

\usepackage{setspace}
\usepackage[utf8]{inputenc}
\usepackage{bm}
\usepackage{mathtools,amssymb}
\usepackage{setspace}
\usepackage{esint}
\usepackage{hyperref}
\hypersetup{
    colorlinks=true,
    linkcolor=black,
    filecolor=black,      
    urlcolor=cyan,
    citecolor=black
    }
\usepackage{tikz}
\usepackage{dsfont}
\usepackage{relsize}
\usepackage{url}
\usepackage{xcolor}
\usepackage{graphicx}
\usepackage{mathrsfs}
\usepackage[shortlabels]{enumitem}
\usepackage{lineno}
\usepackage{amsmath}
\usepackage{enumitem}
\usepackage{amsthm} 
\usepackage{fullpage} 
\usepackage{verbatim}
\usepackage{dsfont}
\usepackage{tikz}
\usetikzlibrary{positioning}
\usetikzlibrary{calc,intersections,through,backgrounds,arrows.meta}

\allowdisplaybreaks

\mathtoolsset{showonlyrefs}

\graphicspath{{images/}}

\newtheorem{theorem}{Theorem}[section]
\newtheorem{lemma}[theorem]{Lemma}
\newtheorem{definition}[theorem]{Definition}
\newtheorem{corollary}[theorem]{Corollary}

\newtheorem{proposition}[theorem]{Proposition}

\newtheorem{remark}[theorem]{Remark}
\newtheorem{example}[theorem]{Example}

\title[Fractional random walks]{An inverse problem for fractional random walks on finite graphs}

\author{Giovanni Covi\textsuperscript{1}}
\author{Matti Lassas\textsuperscript{2}}

\thanks{\textsuperscript{1}\underline{Corresponding author}. Department of Mathematics and Statistics, University of Jyv\"askyl\"a, PO Box 35, 40014 Jyv\"askyl\"a, Finland. Email: giovanni.g.covi@jyu.fi}
\thanks{\textsuperscript{2}Department of Mathematics and Statistics, University of Helsinki, Pietari Kalmin katu 5 (Exactum), Helsinki, Finland. Email: matti.lassas@helsinki.fi}

\newcommand{\R}{{\mathbb R}}
\newcommand{\Z}{{\mathbb Z}}
\newcommand{\N}{{\mathbb N}}
\newcommand{\Q}{{\mathbb Q}}

\usepackage{amsmath}
\usepackage{bm} 

\begin{document}

\setlength{\baselineskip}{24.09pt}

\maketitle
\begin{abstract}
We study an inverse problem on a finite connected graph $G=(X,E)$, on whose vertices a conductivity $\gamma$ is defined. Our data consists in a sequence of partial observations of a fractional random walk on $G$. The observations are partial in the sense that they are limited to a fixed, observable subset $B\subseteq X$, while the random walk is fractional in the sense that it allows long jumps with a probability $P$ decreasing as a fractional power of the distance along the graph. The transition probability $P$ also depends on $\gamma$. We show that this kind of random walk data allows for the determination of a gauge class to which the transition probability matrix $P$ belongs, which we discuss. Moreover, we show that if the transition probability matrix $P$ is itself known, then the amount of vertices $|X|$, the edge set $E$ and the conductivity $\gamma$ (up to a positive factor) can be recovered. We also show a characterization of the random walk data in terms of the corresponding transition matrices $P$, which highlights a new surprising nonlocal property. This work is motivated by the recent strong interest in the study of the fractional Calder\'on problem in the Riemannian setting. 
\end{abstract}


\section{Introduction}

In this paper we study a geometric inverse problem on a finite, connected graph $G=(X,E)$ of vertices $X$ and edges $E\subseteq [X]^2:= \{A\subseteq X : |A|=2\}$. Given a vertex $x_0\in X$, we consider a random walk $\{H^{x_0}_t\}_{t\in\N}$ on $G$ starting at $x_0$. This is a sequence of random variables $H^{x_0}_t:\Omega\rightarrow X$ defined on a probability space $\Omega$ and with values in the set of vertices $X$, indicating the position of a random walker at time $t\in\N$. Because we assume that the random walk starts at the vertex $x_0$, we let $H^{x_0}_0=x_0$. We also define a function $\gamma: X\rightarrow (0,\infty)$, which we call \emph{conductivity}, and make the assumption that the \emph{transition probability} $P(x,y)$ of the random walk $H^{x_0}_t$ between distinct, not necessarily adjacent vertices $x,y\in X$ depends on $\gamma$ in the following way:
$$ P(x,y):= \frac{1}{m(x)} \frac{\gamma^{1/2}(x)\gamma^{1/2}(y)}{d_G(x,y)^{\alpha}}. $$
Here $d_G$ indicates the distance between $x$ and $y$ along the graph, $m(x)$ is a normalization constant, and the power $\alpha>0$ is assumed to be fixed and known (see Remark \ref{significance}). We allow for measurements of the random walk only on a fixed set of vertices $B\subseteq X$, which we call \emph{observable}. This means that for every discrete time step $t\in\N$, we know either that the random walker is not in $B$, or we know at what vertex $x\in B$ it is found. The inverse problem we study consists in recovering from this information the total amount of vertices $|X|$, the edge set $E$, and the conductivity $\gamma$. \\

\begin{figure}[!h] \label{figure-intro}
\centering
\begin{tikzpicture}
\filldraw[blue] (0,0) circle (2pt);
    \node[shape=circle,draw=blue] (A) at (0,0) {};
\filldraw[blue] (0,2) circle (2pt);
    \node[shape=circle,draw=blue] (B) at (0,2) {};
\filldraw[blue] (2,3) circle (2pt);
    \node[shape=circle,draw=blue] (C) at (2,3) {};
    \node[shape=circle,draw=black] (D) at (1.3,1) {};
    \node[shape=circle,draw=black] (E) at (2,-1) {};
    \node[shape=circle,draw=black] (F) at (2.7,1) {} ;
    \node[shape=circle,draw=black] (G) at (4,0) {};
    \node[shape=circle,draw=black] (H) at (4,2) {};

    \path [-] (A) edge node[left] {} (B);
    \path [-](B) edge node[left] {} (C);
    \path [-](A) edge node[left] {} (D);
    \path [-](D) edge node[left] {} (C);
    \path [-](A) edge node[right] {} (E);
    \path [-](D) edge node[left] {} (E);
    \path [-](E) edge node[right] {} (F); 
    \path [-](E) edge node[right] {} (G);
    \path [-](F) edge node[left] {} (H);
\end{tikzpicture}
\caption{An example of graph as in our assumptions. The observable set $B$ is represented in blue. The total amount of vertices, as well as all the edges and the conductivity, are to be recovered from random walk data.}
\end{figure}
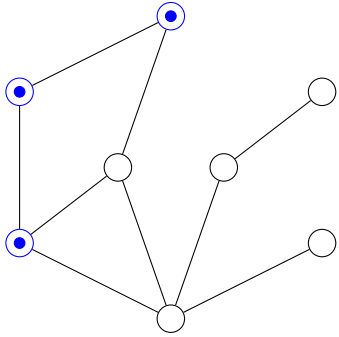

We were inspired by the recent work \cite{BILL21}, which considered an inverse problem on a graph in the context of the discrete heat equation. However, the operators considered there are local, and thus the techniques used in this article (which are adapted for nonlocal, fractional operators) differ substantially from the ones already present in the literature.

Our paper considers an inverse problem on a weighted graph, in which the edge structure and weights have to be recovered from random walk data. Inverse problems
for random walks have been object of study in \cite{50, 51, 82, 85}. Moreover, numerous applications in optical
tomography \cite{60, 61, 62, 63, 84}, network tomography \cite{64, 92}, electrical resistor networks
\cite{42, 71, 80} and neuroscience \cite{8} have been considered. In all these works, different types of random walk measurements are used in order to recover the transition matrix, but the edge set is always assumed to be known a priori. The works \cite{9, 79, 81} have studied the relation between edge structure and random walk properties, but the inverse problem of recovering the edge set itself has proved much harder to solve. Results have been obtained for trees, while more general types of graphs were considered only recently. For these kinds of results we refer to \cite{3}. Inverse problems on graphs are known to be connected to the classical Calder\'on problem, of which they represent a discretization. The works \cite{21, 38, 44, 46, 77} have studied electrical impedance tomography for resistor networks, which consists in recovering the edge structure and weights (conductivities) of a graph from electrical measurements performed at an accessible subset of vertices. The works in this field were focused either on determining the resistor values of known networks, or on recovering equivalence classes of networks corresponding to given boundary data \cite{39, 40, 45}.

The problem we study is fractional and nonlocal in nature. There has recently been very strong interest in the field of inverse problems for fractional operators, ever since the fractional Calder\'on problem was introduced in the seminal work \cite{GSU20}. Many different aspects involving uniqueness have been explored, among which low regularity \cite{RS20}, higher order \cite{CMRU22}, single measurement reconstruction \cite{GRSU20}, and monotonicity methods \cite{HL20, HL20a}. There has also been interest in other equations of fractional kind, as in \cite{RS20a,LLR20,CLR20,BGU21,L21,L23}. In particular, \cite{C20} and the subsequent works \cite{C20a,C22,CRZ22,CRTZ24} have studied (global) uniqueness and stability for the fractional conductivity equation, while \cite{CdHS22} has introduced an inverse problem for a related fractional elasticity operator. Besides uniqueness, the stability and instability properties of the fractional Calder\'on problem have also been object of intensive study \cite{RS18,R21,KRS21,BCR24}. Part of the interest in the fractional Calder\'on problem is due to its relation to its classical counterpart, which in its high dimensional, anisotropic version is one of the main and longest-standing open problems in the field. The recent works \cite{GU21, CGRU23} have presented a scheme following which uniqueness properties for the fractional Calder\'on problem can be deduced from the corresponding uniqueness results in the local case. On this topic we also refer to \cite{LLU23,Lin23}. 

The present article finds its natural place in the literature as a first step towards the goal of studying the inverse problem for the fractional conductivity equation (as introduced in \cite{C20}) in the Riemannian case. We will follow a similar homogenization technique as the one used in \cite{Val09,C20} for the definition of the fractional Laplacian and the conductivity operators in the Euclidean case. As discussed in our Section 2, these nonlocal operators can be obtained through the study of the continuous limit of a nonlocal random walk on a lattice. In a similar fashion, we plan to approximate a Riemannian manifold with a special graph, and obtain the Riemannian fractional conductivity operator as a continuous limit related to a nonlocal random walk on this graph. We discuss this plan in more detail in the final Section 8, and reserve the study of this procedure to future works. The article \cite{FGKU21} has also given rise to many recent efforts in the study of a version of the fractional Calder\'on problem in the Riemannian case. We must however highlight the fact that the operators considered in \cite{FGKU21} and the subsequent papers, as well as the framework from which they originate and the techniques employed in their study, differ substantially from the ones we treat here. While the ones in the referred articles are obtained spectrally, the fractional operators we are interested in are obtained from the nonlocal vector calculus of \cite{DGLZ12, DGLZ13}, as presented in Section 2.\\

We next present the main results of this paper. To this end, we start by defining what we mean by an \emph{admissible graph}.

\begin{definition}[Admissible graph]
    We say that a graph $G=(X,E)$ is \emph{admissible (with respect to an observable set $B\subseteq X$)} when it satisfies the two following conditions:
    \begin{enumerate}[label=(A\arabic*)]
        \item \label{adm-cond-1} The block $P_{12}$ of the transition matrix corresponding to a fractional conductivity random walk $H^{x_0}_t$ on $G$ has full column rank, and
        \item \label{adm-cond-2} the leaf set $L$ of $G$ is such that $|L|\geq 2$ and $\max_{\ell \in L}e(\ell) >3$.
    \end{enumerate}
\end{definition}

We refer to Sections \ref{sec:preliminaries} and \ref{sec:random-walk-data} for the precise definitions of the transition matrix $P$, the fractional conductivity random walk $H^{x_0}_t$, the leaf set $L$ and the eccentricity function $e$. The splitting of the transition matrix $P$, which produces the block $P_{12}$, is done with respect to the observable set $B$, and will be defined in equation \eqref{eq:P-splitting}. Already at this point we observe that a necessary condition for a graph $G$ to verify \ref{adm-cond-2} is that $|X|\geq 5$, since $G$ must have at least a minimal path of length $4$ between two leaves. The graph $G$ can also never be a \emph{star}, that is, there can not be a vertex $x\in X$ such that $y\sim x$ for all $y\in X\setminus\{x\}$, since in this case the eccentricity of every vertex would be at most $2$. Moreover, we notice that condition \ref{adm-cond-1} does not actually depend on the conductivity $\gamma$, and is therefore a property of the graph $G$ (see Remark \ref{remark-admissible-2}). We also observe that, while the admissibility assumption is sufficient for our arguments, it is yet unknown whether it is also necessary (see Section \ref{sec:open-problems}).

The random walk data corresponding to a given graph $G$ and conductivity $\gamma$ (and therefore to a given transition matrix $P$) is given in the form of a sequence of matrices $$\Lambda_\omega(P):=\{ f(P), f(P^2), f(P^3), ... \},$$
where the function $f$ maps a matrix belonging to $\R^{(N+M)\times(N+M)}$ to its upper left $\R^{N\times N}$ block. Here $N,M\in\N$ respectively represent the amounts of observable and unknown vertices. We also make use of a reduced form or random walk data: For all $K\in\N$, we let
$$\Lambda_K(P):=\{ f(P), f(P^2), ..., f(P^K)  \}.$$
The precise definitions of the random walk data will be given in Section 3. For now, it suffices to know that $\Lambda_\omega(P)$ corresponds to knowing the transition probabilities in any amount of jumps between any pair of observable vertices, while $\Lambda_K(P)$ indicated the same transition probabilities, but in at most $K$ jumps.

With the above definitions at hand, we can present the main result of the paper:

\begin{theorem}[Uniqueness and reconstruction]\label{th:main}
    Let $G=(X,E)$ be an admissible graph, $B\subseteq X$ the set of observable vertices, and $\gamma : X\rightarrow (0,\infty)$ an unknown conductivity. Consider the fractional conductivity random walk $H^{x_0}_t$ on $G$, and let $P$ be the corresponding transition matrix. Let the same assumptions hold for $\tilde G= (\tilde X, \tilde E)$, $\tilde B\subseteq \tilde X$, $\tilde \gamma$ and $\tilde P$, with $|B|=|\tilde B|$.
    
    If $\Lambda_3(P) = \Lambda_3(\tilde P)$ holds, then $|X|=|\tilde X|$, and the amount of vertices can be computed from the random walk data. Moreover, there exists a row-stochastic matrix $A\in GL_M$ such that $P=g(A,\tilde P)$, where $g$ is the group action $g: GL_M\times \R^{(N+M)\times (N+M)}\rightarrow \R^{(N+M)\times (N+M)}$ given by
$$ g(A,\tilde P):= (Id_N\oplus A) \tilde P (Id_N\oplus A^{-1} ). $$ 

If in particular $P=\tilde P$, then $E = \tilde E$, the conductivities $\gamma$ and $\tilde \gamma$ differ by a positive factor, and both $\tilde E$ and $\tilde \gamma$ (up to a positive factor) can be reconstructed from the random walk data.
\end{theorem}

We deduce the proof of Theorem \ref{th:main} by means of the following procedure. First, we reconstruct the transition matrix from the raw random walk data, i.e. the transition probabilities in $1,2$ and $3$ steps between observable vertices. In this respect, we obtain the following characterization result (Section \ref{sec:data-char} clarifies the meaning of all the symbols involved in the statement):

\begin{theorem}[Characterization of the random walk data]\label{characterization-new}
Let $N,M,\tilde M\in \N$, and assume $P\in \R^{(N+M)\times(N+M)}$ and $\tilde P\in \R^{(N+\tilde M)\times(N+\tilde M)}$. If $\Lambda_3(P)=\Lambda_3(\tilde P)$, then there exist $r\in\N$ and two matrices $A\in \R^{r\times M}$ and $\tilde A \in \R^{r\times \tilde M}$ of full row rank such that 
 $$(Id_N\oplus A) P (Id_N\oplus A^+) = (Id_N\oplus \tilde A) \tilde P (Id_N\oplus \tilde A^+).$$ Moreover, if $P_{12}, \tilde P_{12}$ have full column rank and $P_{21}, \tilde P_{21}$ have full row rank, then the three following statements are equivalent:
\begin{enumerate}
    \item $\Lambda_3(P)=\Lambda_3(\tilde P)$,
    \item $\tilde M = M$, and there exists $A\in GL_M$ such that
$$ P = (Id_N\oplus  A) \tilde P (Id_N\oplus  A^{-1} ),$$
    \item $\Lambda_\omega(P)=\Lambda_\omega(\tilde P)$.
\end{enumerate} 
\end{theorem}

We observe that the random walk data is not itself sufficient in order to determine a unique transition matrix: One rather obtains that there exists a whole gauge class of matrices which verify the random walk data in the form given in the above Theorem \ref{characterization-new}. Section \ref{sec:gauge} is devoted to the study of the gauge class. The most surprising implication of the characterization Theorem \ref{characterization-new} is that random walk data in terms of transition probabilities between observable points for $k>3$ is \emph{redundant}, i.e. it does not add new information useful for the reconstruction. We have also shown that the random walk data considered in the characterization Theorem \ref{characterization-new} is all \emph{necessary} for the reconstruction, thus fixing the amount of interesting steps to observe to exactly $3$. We give an intuitive interpretation of this result in Remark \ref{rmk:intuition}. \\

The second step in the proof of our main Theorem \ref{th:main} consists in recovering the edge set $E$ and conductivity $\gamma$ from the knowledge of the transition matrix $P$. This recovery is based on a delicate comparison of the transition probabilities between adjacent vertices, particularly the leaves of the graph and their neighbours. After having identified the leaves (both observable and not), we use this knowledge to determine the entire structure of the edge set. It is easy then to recover the conductivity from this and the transition probabilities. \\

The remaining part of the article is organized as follows. Section 2 contains the preliminaries about linear algebra, random walks, and fractional operators which will be necessary in the paper. The random walk data is described in Section 3. Section 4 contains the proof of the characterization Theorem \ref{characterization-new}, which states that the random walk data determines a gauge class of matrices to which $P$ belongs.
The resulting gauge class is described in detail in Section 5. Sections 6 and 7 are dedicated to the solution of the inverse problem, i.e. the proof of the main Theorem \ref{th:main}. Finally, we have collected some open questions and conjectures related to our study in the final Section 8.

\section{Preliminaries}\label{sec:preliminaries}

We dedicate this section to the presentation of some concepts from linear algebra, graph theory, and nonlocal vector calculus, which we use in the coming arguments.

\subsection{Linear algebra}

Here we define the concepts from linear algebra used later. We start from the \emph{Moore-Penrose inverse} of a matrix $A$, which is a generalization of the concept of inverse matrix (see \cite{Pen55}):

\begin{definition}[Moore-Penrose inverse] Let $m,n\in\N$ and $A\in \R^{m\times n}$. The \emph{Moore-Penrose inverse} (or \emph{pseudoinverse}) of $A$ is the unique matrix $A^+\in\R^{n\times m}$ such that the four following conditions hold:
    $$AA^+A=A, \qquad A^+AA^+=A^+, \qquad (AA^+)^T=AA^+, \qquad (A^+A)^T=A^+A.$$
If $A$ has full rank, that is rank$(A)=\min\{m,n\}$, then either $AA^T$ or $A^TA$ is invertible, and
$$A^+=\begin{cases}
    A^T(AA^T)^{-1} , \quad\mbox{ if } m\leq n  \\ (A^TA)^{-1}A^T, \quad\mbox{ if } n\leq m
\end{cases}.$$
If $m\leq n$ (resp. $n\leq m$), then $A^+$ is a right (resp. left) inverse of $A$. If $m=n$, then $A$ is invertible and $A^+=A^{-1}$.
\end{definition} 

It is possible to obtain the Moore-Penrose inverse of a matrix $A$ via the so called \emph{full rank factorization} (see \cite[Theorems 1 and 2]{PO99}), which we define next.

\begin{definition}[Full rank factorization] Let $m,n,r\in\N$, and assume that $A\in \R^{m\times n}, F\in \R^{m\times r}$ and $G\in \R^{r\times n}$ are such that 
$$ \mbox{rank}(A)=\mbox{rank}(F)=\mbox{rank}(G)=r>0, \qquad A=FG.$$
Then we say that $F$ and $G$ give a \emph{full rank factorization} of $A$. 

Moreover, $F,G$ and $F_1,G_1$ give two full rank factorizations of the same matrix $A$ if and only if there exists an invertible matrix $R\in \R^{r\times r}$ such that $F_1 = FR$ and $G_1 = R^{-1}G$.
\end{definition} 

It should be noted that a full rank factorization always exists for any matrix of positive rank, as proved for example in \cite[Section 3]{PO99}. We also present the concept of \emph{transitive matrix}, as it appears in \cite{FRS99}:

\begin{definition}[Transitive matrix]
    Let $n\in \N$ and $A\in \R^{n\times n}$. We say that $A$ is \emph{transitive} if $A_{ik}A_{kj}=A_{ij}$ for all $i,j,k\in\{1,...,n\}$.
\end{definition}

\subsection{Random walks on graphs} Next, we present the main ideas which we need from graph theory. Let $G=(X,E)$ be a finite, connected, undirected graph whose vertices belong to $X$ and whose edges belong to $E\subseteq [X]^2:= \{A\subseteq X : |A|=2\}$. When $x,y\in X$ share an edge, that is when $\{x,y\}\in E$, we say $x\sim y$. We define the \emph{leaf set} $L\subseteq X$ of $X$ as $$L:= \{x\in X: \exists ! y\in X : y\sim x\},$$
that is the set of all vertices with exactly one neighbour. Observe that there exist graphs, such as any totally connected graph with at least $3$ vertices, for which $L=\emptyset$. \\

Since $G$ is connected, for all pairs $x,y\in X$ with $x\neq y$ there exists at least one finite sequence of vertices $\{x_j\}_{j=0}^n$ such that $$ x_0=x, \qquad x_n=y, \qquad x_{j-1}\sim x_{j} \quad \mbox{for all } j\in\{1,...,n\}.  $$ This sequence is called a \emph{path from $x$ to $y$}, and its \emph{length} is the natural number $n$. With this in mind, we can define the \emph{distance} $d_G(x,y)$ of two vertices $x,y\in X$ as $0$ when $x=y$, and otherwise as the length of the shortest path from $x$ to $y$. We also define the function $e: X\rightarrow \N$ given by $$ e(x) := \max_{y\in X} d_G(x,y),$$ which is called \emph{eccentricity}. It associates to each point $x\in X$ the distance to the point $y\in X$ which is furthest from $x$. \\

Next, we will consider a discrete-time random walk on the graph $G$. Given $x_0\in X$, we let $H^{x_0}_t, t\in\N$ be a discrete-time Markov chain with starting point $x_0$ and state space $X$. We will refer to $H^{x_0}_t$ as the \emph{random walk}. For any $x,y\in X$, we define the \emph{transition probability} $P(x,y)$ from $x$ to $y$ as 
$$P(x,y):= \mathbb P(H^{x_0}_{t+1}=y|H^{x_0}_t=x),$$
which is assumed to be independent of $t$. When seen as a matrix whose rows and columns are indexed in the set $X$, $P$ constitutes the \emph{transition matrix} of the Markov chain $H^{x_0}_t$. We remark that the transition matrix $P$ is row-stochastic, but it needs not to be symmetric or column-stochastic. Recall that a square matrix $A\in\R^{n\times n}$ is said to be \emph{row-stochastic} when all its rows sum to $1$, that is, when the vector $\mathbf 1\in \R^{n}$ composed of all $1$'s is a right-eigenvector of $A$ corresponding to the eigenvalue $1$ (see also the subsection Notation):
$$ A\mathbf 1 = \mathbf 1.$$
\noindent In fact, even if the underlying mechanism that produces the random walk is symmetric with respect to the vertices $x$ and $y$ (as it will in fact be the case for our fractional conductivity), this symmetry is lost at the moment of the row-wise normalization of the probabilities.
Since $P$ is the transition matrix of a Markov chain, it is easy to show by induction that it verifies the following Markov chain property: for all $k\in\N$ and $x,y\in X$, the transition probability from $x$ to $y$ in $k$ steps is given by the $(x,y)$ element of the $k$-th power of $P$, that is
$$P^k(x,y)= \mathbb P(H^{x_0}_{t+k}=y|H^{x_0}_t=x).$$
We say that a random walk $H^{x_0}_t$ is \emph{allowed to stay at a vertex} $x\in X$ if the transition probability $P(x,x)$ does not vanish, that is if the probability to jump back to the same vertex $x$ in one step is not $0$. Similarly, we simply say that a random walk $H^{x_0}_t$ is \emph{allowed to stay} if $P(x,x)\neq 0$ for all $x\in X$. 

\subsection{Fractional conductivity}
We shall now consider a specific kind of random walk on a graph $G$ arising from the study of fractional conductivity, following \cite{C20}. Let $n\in\N$ and $s\in(0,1)$. Consider a scalar function $\gamma\in C^\infty(\R^n)$, which we call \emph{conductivity}. The \emph{fractional conductivity operator} $\mathbf C^s_\gamma$ is weakly defined as
$$ \langle\mathbf C^s_\gamma u,v\rangle :=\langle \gamma^{1/2}(x)\gamma^{1/2}(y)\nabla^s u,\nabla^s v \rangle, $$
where $u,v\in H^s(\R^n)$ are Sobolev functions. The \emph{fractional gradient} $\nabla^s$ is defined at first for $u\in C^\infty_c(\R^n)$ as $$\nabla^su(x,y) := \frac{C_{n,s}^{1/2}}{\sqrt 2} \frac{u(y)-u(x)}{|x-y|^{n/2+s+1}}(x-y), \qquad\mbox{ where } \quad C_{n,s}:= \frac{4^s\Gamma(n/2+s)}{\pi^{n/2}|\Gamma(-s)|},$$
and then extended to act on $H^s(\R^n)$ by density. Observe that the fractional gradient is a 2-points function, in that it maps $\nabla^s : H^s(\mathbb R^n)\rightarrow L^2(\mathbb R^{2n})$. The fractional conductivity operator was introduced in \cite{C20} based on the nonlocal vector calculus studied in \cite{DGLZ12, DGLZ13}, as a generalization of the fractional Laplacian operator $(-\Delta)^s$. For a smooth, compactly supported function $u$, the two operators can be written as $$(-\Delta)^su(x):= C_{n,s}\int_{\R^n}\frac{u(x)-u(y)}{|x-y|^{n+2s}}dy, \qquad  \mathbf C^s_\gamma u(x) = C_{n,s}\int_{\R^n} \gamma^{1/2}(x)\gamma^{1/2}(y)\frac{u(x)-u(y)}{|x-y|^{n+2s}}dy.$$ Observe that the constant $C_{n,s}$ verifies the property $$ \lim_{s\rightarrow 1^-}\frac{C_{n,s}}{s(s-1)} = \frac{4n}{\omega_{n-1}},$$
which ensures that $\lim_{s\rightarrow 1^-}(-\Delta)^su = -\Delta u$. If we define the \emph{fractional divergence} $(\nabla\cdot)^s:=(\nabla^s)^*$ as the adjoint of the fractional gradient, we can prove the property $(\nabla\cdot)^s\nabla^s = (-\Delta)^s$, which is familiar from the classical, non-fractional case.

The main result of \cite{C20} discusses an inverse problem in $\mathbb R^n$, namely whether the conductivity function $\gamma$ can be uniquely determined from nonlocal Dirichlet-to-Neumann data $\Lambda_\gamma$ taken in the exterior of an open, bounded set $\Omega$ in which the fractional conductivity equation $\mathbf C^s_\gamma u=0$ is assumed to hold. The cited paper indeed shows that uniqueness in the recovery of $\gamma$ from $\Lambda_\gamma$ holds, as an effect of the unique continuation property of the fractional Laplacian proved in \cite{GSU20}, and of a reduction argument for the fractional conductivity equation reminiscent of the classical one by Liouville, which holds in the non-fractional case. It was also shown in \cite[Section 5]{C20} that the fractional conductivity operator naturally arises as a continuous limit of a discrete random walk with weights and long jumps (see also \cite{Val09} for the same argument for the fractional Laplacian operator). Fix $h>0$, and consider a random walk on the graph $(h\Z^n,E)$ subject to time increments in $h^{2s}\N$. Here we let $x\sim y$ for $x,y\in h\Z^n$ if and only if $x-y=he_j$ for some $j\in\{1,...,n\}$, where $e_j$ is the vector whose $j$-th coordinate is $1$ and the others are $0$. Given $x,y\in h\Z^n$, the transition probability from $x$ to $y$ is obtained by normalizing
$$C(x,y):=\begin{cases}
    \gamma(x)c(x)  &, x=y \\
    h^{n+2s}\frac{\gamma^{1/2}(x)\gamma^{1/2}(y)}{|x-y|^{n+2s}}  &, x\neq y
\end{cases}$$
with respect to $x$, that is
$$ P(x,y):= \frac{C(x,y)}{m(x)}, \qquad m(x):= \sum_{y\in h\Z^n}C(x,y). $$
Here $P(x,x)=\frac{\gamma(x)}c(x){m(x)}$ represents the (possibly non-zero) probability that the random walk stays at $x$ after one time step. If $u(x,t)$ is the probability that the random walk is at $x\in h\Z^n$ at time $t\in h^{2s}\Z$, then we have
\begin{align*}
    m(x)\partial_tu(x,t)=\lim_{h\rightarrow 0}h^{n}\sum_{y\in h\Z^n\setminus\{x\}}\gamma^{1/2}(x)\gamma^{1/2}(y)\frac{u(x,t)- u(y,t)}{|x-y|^{n+2s}}\approx \mathbf C^s_\gamma u(x,t),
\end{align*}
from which we see that the stationary state must solve $\mathbf C^s_\gamma u(x)=0$.\\

Let us now consider a graph $G=(X,E)$ and a conductivity $\gamma: X\rightarrow (0,\infty)$ defined on the vertices of the graph. In analogy to the previous case, we let
$$ C(x,y):=\begin{cases}
    \gamma(x)c(x), & x=y \\
    \frac{\gamma^{1/2}(x)\gamma^{1/2}(y)}{d_G(x,y)^{n+2s}},  & x\neq y
\end{cases} , \qquad m(x):= \sum_{y\in X}C(x,y), \qquad P(x,y):= \frac{C(x,y)}{m(x)}. $$
We remark that for the graph $(h\Z^n,E)$ we did not use the distance along the graph explicitly, but we would have obtained the same limit result due to the equivalence of all norms in finite-dimensional vector spaces. The \emph{graph fractional conductivity operator} is then defined as
$$ \mathbf C^s_{\gamma,G} u(x) = C_{n,s} \sum_{y\in X\setminus\{x\}} \gamma^{1/2}(x)\gamma^{1/2}(y)\frac{u(x)- u(y)}{d_G(x,y)^{n+2s}}, $$
and in particular the \emph{graph fractional Laplacian} is
$$ (-\Delta)^s_{G} u(x) = C_{n,s} \sum_{y\in X\setminus\{x\}} \frac{u(x)- u(y)}{d_G(x,y)^{n+2s}}. $$
We see that these are not examples of the graph Laplacian, due to the fact that our random walks allow long jumps. Recall that, given two fixed functions $\mu: X\rightarrow \R$ and $g:X^2\rightarrow \R$ with $g$ symmetric, the \emph{graph Laplacian} is defined as  $$\Delta_{\mu,g,G}u(x) = \mu(x)^{-1}\sum_{y\sim x} g(x,y)(u(x)-u(y)).$$ 
However, if we define the new completely connected graph $G'$ on the same vertices $X$, we see that $ \mathbf C^s_{\gamma,G}= C_{n,s}\Delta_{1, \frac{\gamma^{1/2}(x)\gamma^{1/2}(y)}{d_G(x,y)^{n+2s}} ,G'} $, where the distance function is crucially still taken with respect to the original graph $G$.

\subsection{Notation}
We conclude the preliminaries with a short clarification of our notation. 

Let $X,Y$ be \emph{finite} sets. After having fixed a particular order for the elements of $X$, it is possible to understand any function $f: X\rightarrow \R$ as a vector in $\R^{|X|}$, indexed by the elements of $X$. Analogously, after having fixed an order for the elements of $Y$, one can understand any function $F:X\times Y\rightarrow \R$ as a matrix belonging to $\R^{|X|\times |Y|}$. We use this identification throughout the paper. In particular, the conductivity $\gamma$ is seen as a vector, and the transition matrix $P$ is also understood as a function of two variables.  

We further define the symbols we use for matrix operations. Let $n,m,\nu,\mu \in\N$, and consider the matrices $A\in R^{n\times m}$, $B\in \R^{\nu\times\mu}$. We indicate by $A\oplus B$ the matrix in $\R^{(n+\nu)\times (m+\mu)}$ defined block-wise as
$$A\oplus B = \left(\begin{array}{cc}
    A & 0 \\
    0 & B
\end{array}\right).$$
If $A,B\in \R^{n\times m}$, the Hadamard product $A\odot B \in \R^{n\times m}$ is given by $$ (A\odot B)_{ij} = A_{ij}B_{ij}, \qquad \mbox{for all } i\in\{1,...,n\}, j\in\{1,...,m\}.$$
In the assumption that all the entries of $B$ are different from $0$, we also define the Hadamard quotient $A\oslash B \in \R^{n\times m}$ as $$ (A\oslash B)_{ij} = A_{ij}/B_{ij}, \qquad \mbox{for all } i\in\{1,...,n\}, j\in\{1,...,m\}.$$
Moreover, we define the function diag $ : \R^n \rightarrow \R^{n\times n}$ which associates to any vector $v\in \R^{n}$ the square matrix diag$(v)$ having $v$ as principal diagonal and $0$ for all the non-diagonal entries.

We indicate by $Id_n$ the identity matrix in $\R^{n\times n}$ (we may omit the suffix $n$ when the dimension is clear from the context). We also let $\mathbf 1$, $\mathbf 0$ be the vectors (or more in general the matrices) whose all entries are respectively $1$ and $0$. When necessary, we will indicate the dimensions of these objects with suffixes. 

\section{ Partial random walk data }\label{sec:random-walk-data}
In this section we will write the random walk data in a more manageable, algebraic form. This will then be used in the coming sections, where the set of all possible transition matrices which agree with the random walk data will be characterized. \\

Let $B\subseteq X$ be the set of observable vertices, and define $N:=|B|$, $M:=|X\setminus B|$. Let us fix an order for the elements of $X$ in such way that the first $N$ of them are the observable vertices, i.e. the elements of $B$. In order to separate the transition information between the observable vertices belonging to $B$ from the transition information regarding the unobservable vertices belonging to $X\setminus B$, we introduce the following block-splitting of the transition matrix $P\in \R^{(N+M)\times (N+M)}$:
\begin{equation}\label{eq:P-splitting} P=\left(\begin{array}{cc}
    P_{11} & P_{12} \\
    P_{21} & P_{22}
\end{array}\right), \end{equation}
where $P_{11} \in \R^{N\times N}$, $P_{12}\in \R^{N\times M}$, $P_{21}\in \R^{M\times N}$, and $P_{22}\in\R^{M\times M}$. It is now clear that the sub-matrix $P_{11}$ contains the transition probabilities between elements of $B$. We observe at this point that there are of course many different ways of ordering the unobservable vertices from $X\setminus B$. This gives rise to a finite family of equivalent representations for the transition probability $P$: if $\Pi\in \R^{M\times M}$ is any permutation matrix, then $P$ may be equivalently represented as
$$ (Id_N\oplus \Pi)P(Id_N\oplus \Pi^T) =\left(\begin{array}{cc}
    P_{11} & P_{12}\Pi^T \\
    \Pi P_{21} & \Pi P_{22} \Pi^T
\end{array}\right). $$\\
Thus the transition matrices $P$ would more accurately be represented in our case by classes in $\R^{(N+M)\times (N+M)}$, each one determined by applying the above permutations to a given representative. However, this understanding does not substantially improve our analysis of the problem, and therefore we will not mention it further.\\

If we similarly split the matrix $P^k$, where $k\in\N$, by the Markov chain property we obtain that the sub-matrix $(P^k)_{11}\in \R^{N\times N}$ contains the transition probabilities in $k$ steps between pairs of observable vertices in $B$. In general, $(P^k)_{11} \neq (P_{11})^k$.  By block-wise matrix multiplication, we rather have 
   $$ P^2 = \left(\begin{array}{cc}
    (P_{11})^2 + P_{12}P_{21} & P_{11}P_{12} + P_{12}P_{22} \\
    P_{21}P_{11}+P_{22}P_{21} & P_{21}P_{12} + (P_{22})^2 
\end{array}\right), $$
and so
\begin{equation}\label{eq:formula-P2}
    (P^2)_{11} =  (P_{11})^2 + P_{12}P_{21},
\end{equation}
\begin{equation}\label{eq:formula-P3}
(P^3)_{11} = (P_{11})^3 + P_{11}P_{12}P_{21} + P_{12}P_{21}P_{11}+P_{12}P_{22}P_{21}. 
\end{equation}

Let us now consider the partial random walk data in terms of the block matrices we just defined. 
Since we allow for measurements of the random walk on $B$, then for every discrete time step $t\in\N$, we know either that the random walker is not in $B$, or we know at what vertex $x\in B$ it is found. By observing the random walk for $t\rightarrow\infty$, we can compute the probabilities of all events of the kind
$$\mathbb P(H^{x_0}_{t+k}=y|H^{x_0}_t=x)=P^k(x,y),$$ where $x,y\in B$ and $k\in\N$. This amounts to knowing the sub-matrices $(P^k)_{11}$ for all $k\in\N$. From now on, we will therefore assume that the partial random walk data of the inverse problem is given in the form of the following sequence of matrices 
$$ \Lambda_{\omega}(P):=\{ (P^k)_{11} \}_{k\in\N} \subset \R^{N\times N}. $$
For all $K\in\N$ we also define the following reduced version of the partial random walk data:
$$ \Lambda_{K}(P):=\{ (P^k)_{11} \}_{k\in\{1,...,K\}} \subset \R^{N\times N}, $$
where it is assumed that only the transition probabilities between observable vertices in at most $K$ steps are known. \\

In the next section we will consider the problem of recovering the transition matrix $P$ from partial data of the form $\Lambda_{\omega}(P)$. We will observe that the reduced data $\Lambda_{3}(P)$ suffices to prove uniqueness up to a natural gauge, and that all the rest of the data is redundant.

\section{Recovering a transition matrix from partially observable data}\label{sec:data-char}

In this section we want to answer a fundamental question about random walks on graphs, that is, whether the transition matrix $P$ relative to a random walk on a given graph $G=(X,E)$ can be uniquely recovered from transition data relative to a subset of observable vertices of $G$. As a first step in this direction, we prove Theorem \ref{characterization-new}, using an argument inspired by \cite[Theorem 2]{PO99}. It should however be noted that, while Theorem \ref{characterization-new} completely characterizes the set of matrices $P \in R^{(N+M)\times (N+M)}$ such that $\Lambda_{\omega}(P) = \Lambda_{\omega}(\tilde P)$ (or equivalently $\Lambda_{3}(P) = \Lambda_{3}(\tilde P)$), it says nothing about whether $P$ is a probability matrix, or whether $P$ arises as the transition matrix of a fractional random walk on a graph. We will consider these questions in the coming sections.

\begin{proof}[Proof of Theorem \ref{characterization-new}]
By equation \eqref{eq:formula-P2} we deduce $$ P_{12}P_{21} = (P^2)_{11}-(P_{11})^2, $$
and therefore from \eqref{eq:formula-P3} we have
\begin{align*}
    P_{12}P_{22}P_{21} & = (P^3)_{11} - (P_{11})^3 - P_{11}P_{12}P_{21} - P_{12}P_{21}P_{11} 
    \\ & = 
    (P^3)_{11} - (P_{11})^3 - P_{11}((P^2)_{11}-(P_{11})^2) - ((P^2)_{11}-(P_{11})^2)P_{11}
     \\ & = 
    (P^3)_{11} - P_{11}(P^2)_{11} - (P^2)_{11}P_{11}+(P_{11})^3.
\end{align*} 
Using the assumption $\Lambda_3(P)=\Lambda_3(\tilde P)$ we obtain the following identities of matrices in $\R^{N\times N}$:
$$ P_{12}P_{21} =  \tilde P_{12}\tilde P_{21}, \qquad  P_{12}P_{22}P_{21} =  \tilde P_{12}\tilde P_{22}\tilde P_{21}.  $$

Let $r:= \mbox{rank}( P_{12}P_{21} )$ (observe that $r\leq N$), and consider a full-rank decomposition $R_1R_2 = P_{12}P_{21}$ of the matrix $P_{12}P_{21}$, where $R_1\in \R^{N\times r}$ and $R_2\in \R^{r\times N}$ have full rank. We have  $$R_1^+R_1 = Id_r = R_2R_2^+.$$ Define the matrix $A:= R_1^+P_{12}\in\R^{r\times M}$, which has full row rank since $R_2 =R_1^+R_1R_2 = R_1^+P_{12}P_{21} = AP_{21}$ and thus
$$ r= \mbox{rank}(R_2) = \mbox{rank}(AP_{21}) \leq \mbox{rank}(A) \leq r. $$
Therefore, $A$ admits a (unique) right-inverse $A^+$, and since it holds $Id_r = R_2R_2^+ = AP_{21}R_2^+$ it must be $A^+ = P_{21}R_2^+$. This gives also $R_1 = P_{12}A^+$. We can argue similarly for $\tilde A:= R_1^+\tilde P_{12}\in \R^{r\times \tilde M}$, thus obtaining $R_2 = \tilde A\tilde P_{21}$ and $R_1 = \tilde P_{12}\tilde A^+$. Moreover, we can compute
$$  AP_{22}A^+= R_1^+ P_{12}P_{22}P_{21} R_2^+ =  R_1^+ \tilde P_{12}\tilde P_{22}\tilde P_{21} R_2^+ = \tilde A\tilde P_{22}\tilde A^+. $$
It is now a simple computation to verify that
$$(Id_N\oplus A) P (Id_N\oplus A^+) = (Id_N\oplus \tilde A) \tilde P (Id_N\oplus \tilde A^+),$$
which proves the first part of the statement.

Assume now that $P_{12}, \tilde P_{12}$ have full column rank and $P_{21}, \tilde P_{21}$ have full row rank. We see that $ P_{12}P_{21}$ and $\tilde P_{12}\tilde P_{21}$ constitute two full-rank decompositions of the same matrix, which implies $M=r = \tilde M$. Thus $A$ is a square, full-rank matrix, i.e. it is invertible and $A^+ = A^{-1}$. The same holds true for $\tilde A$. In light of the formula $Id_N\oplus A^{-1} = (Id_N\oplus A)^{-1}$, we conclude   $$ P = (Id_N\oplus \mathcal A) \tilde P (Id_N\oplus \mathcal A^{-1} ),$$
where $\mathcal A := A^{-1}\tilde A$.  This gives the implication $(1)\Rightarrow (2)$. Since $(3)\Rightarrow (1)$ is trivial, we only need to prove $(2)\Rightarrow (3)$. This is an easy computation: for all $k\in\N$ we have
\begin{align*}
    P^k & = \left( (Id_N\oplus \mathcal A) \tilde P (Id_N\oplus \mathcal A^{-1} ) \right)^k
    \\ & = (Id_N\oplus \mathcal A) \left( \tilde P (Id_N\oplus \mathcal A^{-1} )(Id_N\oplus \mathcal A)  \right)^{k-1} \tilde P (Id_N\oplus \mathcal A^{-1} ) 
    \\ & = (Id_N\oplus \mathcal A) \tilde P^k (Id_N\oplus \mathcal A^{-1} ) ,
\end{align*}
and thus $(P^k)_{11} = (\tilde P^k)_{11}$, that is $\Lambda_\omega(P)=\Lambda_\omega(\tilde P)$. Observe that this same result would not hold without the additional assumptions about full rank of the sub-matrices $P_{12},\tilde P_{12}, P_{21},\tilde P_{21}$. In fact, in this case we would obtain
$$ (Id_N\oplus A^+)(Id_N\oplus A) = (Id_N \oplus A^+A), $$
which is in general not the identity, because $A^+$ is the \emph{right} inverse of $A$.
\end{proof}

\begin{remark}\label{rmk:intuition}
It follows from the proof of Theorem \ref{characterization-new} that the transition matrix $P$ can be identified (up to the matrix $ A$) from $\Lambda_k(P)$, with $k=3$. It is natural to ask whether the same result can be obtained for smaller $k$. It is however easily seen that both $P_{11}$ and $(P^2)_{11}$ contain no information about $P_{22}$, and thus recovery is impossible with $k<3$. Intuitively, one can say that $P_{11}$ contains information about couples of vertices in $B^2$, $(P^2)_{11}$ in $B\times (X\setminus B)$, and  $(P^3)_{11}$ in $(X\setminus B)^2$. 
\end{remark}

We complement the above result with the following lemma, which analyzes the linear independence conditions from the previous theorem. The following is a general algebraic result, which holds true also when the transition matrix $P$ does not take the specific form required for fractional conductivity. 

\begin{lemma}
    Let $P\in \R^{(N+M)\times (N+M)}$ be the transition matrix obtained by row-wise normalization of a positive, symmetric matrix $C\in \R^{(N+M)\times (N+M)}$. Then $P_{12}$ has full column rank if and only if $P_{21}$ has full row rank.
\end{lemma}

\begin{proof}
    Observe that the statement is not trivial, given that the matrix $P$ is not in general symmetric. Let $D:= \mbox{diag}(C\mathbf 1)^{-1}$, so that $P= DC$. Let us use for the matrix $C$ the same splitting as in \eqref{eq:P-splitting}, observing that in this case we have $C_{21} = C_{12}^T$, because $C$ is symmetric by assumption. Then
    $$ P_{12} = D_{11}C_{12}, \qquad P_{21} = D_{22}C_{12}^T. $$ Because the matrices $D_{11}, D_{22}$ are invertible, we have
    $$\mbox{rank}(P_{12}) = \mbox{rank}(C_{12}) = \mbox{rank}(C_{12}^T) = \mbox{rank}(P_{21}),$$
    from which it follows that $P_{12}$ has full column rank if and only if $P_{21}$ has full row rank.
\end{proof}

\begin{remark}\label{remark-admissible-2}
In the case of fractional conductivity, we in particular have that
$C=WC'W$, where $W:=\mbox{diag}(\gamma^{1/2})$ and
    $$ C'(x,y) := \begin{cases}
    c(x), & x=y, \\
    d_G(x,y)^{-(n+2s)},  & x\neq y.
\end{cases} $$  
Then we see that $P_{12}= D_{11}W_{11}C'_{12}W_{22}$, and therefore $\mbox{rank}(P_{12}) = \mbox{rank}(C'_{12})$ since the matrices $W_{11}$ and $W_{22}$ are invertible. This implies that $P_{12}$ has full column rank (or equivalently $P_{21}$ has full row rank) if and only if this is true of $C'_{12}$. However, this matrix depends only on the graph $G$, and not on the conductivity. In light of this, we see that it makes sense to assume that a graph $G$ satisfies the admissibility condition \ref{adm-cond-1}, without making mention of $\gamma$.
\end{remark}

We show by means of the next example that there exist plenty of graphs satisfying the admissibility condition \ref{adm-cond-1}. It is worth noticing, however, that a graph verifying \ref{adm-cond-1} must not necessarily be of the kind constructed below.

\begin{example}\label{example-admissible-G} \emph{
    In order to construct an example of a graph satisfying \ref{adm-cond-1}, we will produce a graph $G=(X,E)$ with the property that for all $x\in X\setminus B$ there exists exactly one $b\in B$ such that $d_G(x,b) = 1 < d_G(y,b)$ for all $y\in X\setminus (B\cup\{x\})$. By properly ordering the points of $X$, we see that the square matrix composed of the first $M$ columns of $(d_G(x,y)^{-(n+2s)})_{21}$ has $1$ along the main diagonal and strictly smaller off-diagonal values. It is easily seen that if $M$ is small enough ($M< 1+2^{n+2s}$ will suffice), then this matrix is strictly diagonally dominant, and thus invertible. As a result, the rank of $(d_G(x,y)^{-(n+2s)})_{21}$ is $M$, which proves that its rows (and therefore the rows of $P_{21}$) are linearly independent. \\
    A graph of this kind is easily constructed by attaching to every point $x_i$ of a given graph $J_0$ a new graph $J_i$, and considering the union of the attached graphs as the observable set $B$. More precisely, consider $I\in\N$ and a family of graphs $\{J_i\}_{i=0}^{I}$, with $J_i=(V_i,E_i)$ for all $i=0,...,I$, and $|V_0|=I$. Let $\{x_i\}_{i=1}^{I}$ be an enumeration of the vertices of $J_0$, and choose a vertex $z_i$ for each graph $J_i$, $i=1, ..., I$. We construct a new graph $G$ with vertices $X:=\bigcup_{i=0}^{I} V_i$  and edges given by $E:=E_0 \cup \bigcup_{i=i}^{I} (E_i \cup \{x_i, z_i\} )$. We also declare $B:= \bigcup_{i=i}^{I} V_i$, so that $X\setminus B$ is the set of the vertices of the original graph $J_0$. It is clear that the new graph $G=(X,E)$ has the required property, and thus it satisfies \ref{adm-cond-1} (see Figure \ref{figure-admissible-G})}.
\end{example}    

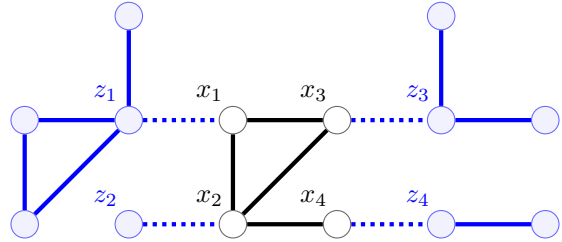
\begin{figure}[!h]
\centering
\begin{tikzpicture}[
Jnode/.style={circle, draw=blue!60, fill=blue!5, thin, minimum size=1mm},
J0node/.style={circle, draw=black!60, thin, minimum size=1mm}
]
\node[J0node][label={[xshift=-.8em, yshift=-.2em]\small $x_1$}] (N1) {};
\node[J0node][label={[xshift=-.8em, yshift=-.2em]\small $x_2$}] (N2) [below=of N1] {};
\node[J0node][label={[xshift=-.8em, yshift=-.2em]\small $x_3$}] (N3) [right=of N1] {};
\node[J0node][label={[xshift=-.8em, yshift=-.2em]\small $x_4$}] (N4) [below=of N3] {};

\node[Jnode][label={[xshift=-.8em, yshift=-.2em]\small \color{blue} $z_1$}] (N5) [left=of N1]  {};
\node[Jnode] (N6) [left=of N5]  {};
\node[Jnode] (N7) [above=of N5] {};
\node[Jnode][label={[xshift=-.8em, yshift=-.2em]\small \color{blue} $z_3$}] (N8) [right=of N3] {};
\node[Jnode] (N9) [above=of N8] {};
\node[Jnode] (U1) [right=of N8] {};
\node[Jnode][label={[xshift=-.8em, yshift=-.2em]\small \color{blue} $z_4$}] (U2) [below=of N8] {};
\node[Jnode] (U3) [below=of U1] {};
\node[Jnode][label={[xshift=-.8em, yshift=-.2em]\small \color{blue} $z_2$}] (U4) [below=of N5] {};
\node[Jnode] (U5) [below=of N6] {};

\draw[ultra thick] (N1.south) -- (N2.north);
\draw[ultra thick] (N1.east) -- (N3.west);
\draw[ultra thick] (N2.east) -- (N4.west);
\draw[ultra thick] (N2.north east) -- (N3.south west);

\draw[ultra thick, blue] (N6.east) -- (N5.west);
\draw[ultra thick, blue] (N8.east) -- (U1.west);
\draw[ultra thick, blue] (U2.east) -- (U3.west);
\draw[ultra thick, blue] (N6.south) -- (U5.north);
\draw[ultra thick, blue] (U5.north east) -- (N5.south west);
\draw[ultra thick, blue] (N7.south) -- (N5.north);
\draw[ultra thick, blue] (N9.south) -- (N8.north);
\draw[ultra thick, blue, dotted] (N5.east) -- (N1.west);
\draw[ultra thick, blue, dotted] (N3.east) -- (N8.west);
\draw[ultra thick, blue, dotted] (U4.east) -- (N2.west);
\draw[ultra thick, blue, dotted] (N4.east) -- (U2.west);
\end{tikzpicture}
\caption{An example of a graph $G=(X,E)$ satisfying \ref{adm-cond-1}, constructed as in Example \ref{example-admissible-G}. Here the graph $J_0$ of vertices $V_0=X\setminus B$ is represented in black, while the observable subgraph $(B,E\setminus E_0)$ is represented in blue. }
\label{figure-admissible-G}
\end{figure}

\section{Discussion of the gauge class}\label{sec:gauge}

Let $G,\gamma$ respectively be a graph and a conductivity, and consider the transition matrix $\tilde P$ corresponding to the pair $G,\gamma$. Let $g: GL_M\times \R^{(N+M)\times (N+M)}\rightarrow \R^{(N+M)\times (N+M)}$ be the group action given by
$$ g(A,\tilde P):= (Id_N\oplus A) \tilde P (Id_N\oplus A^{-1} ), $$
and let $GL_M\cdot \tilde P$ indicate the orbit of $\tilde P$ under $g$, i.e. the set
$$ GL_M\cdot \tilde P := \{ g(A,\tilde P) : A \in GL_M \} .$$
We start by making the following important observation:

\begin{lemma}\label{lem:unique}
    Let $G$ be an admissible graph corresponding to the transition probability $\tilde P$, and assume $P\in GL_M\cdot \tilde P$. Then there exists a unique $A\in GL_M$ such that $P=g(A,\tilde P)$.
\end{lemma}

\begin{proof}
    Assume that $A,B\in GL_M$ verify $g(A,\tilde P) = g(B,\tilde P)$. Then $\tilde P = g(A^{-1}B, \tilde P)$, which implies that
    $$ \tilde P_{21} = A^{-1}B\tilde P_{21}.$$
    Since $G$ is admissible, we know that $\tilde P_{21}$ has full row rank, and so it admits a unique right inverse $\tilde P_{21}^+$. Applying $\tilde P_{21}^+$ to the right in the equality above, we deduce $A^{-1}B = Id$, i.e. $A=B$.
\end{proof}

We have seen in the previous section that if $G$ verifies \ref{adm-cond-1}, then $GL_M\cdot \tilde P$ can be recovered from the random walk data. However, while they satisfy the data, not all the matrices $P\in GL_M\cdot \tilde P$ are transition matrices corresponding to a graph $G'$ and a conductivity $\gamma'$. We have the following:
\begin{proposition}\label{lem:3-cond}
    A matrix $P\in \R^{(N+M)\times (N+M)}$ is obtained by row-wise normalization of a positive (resp. non-negative), symmetric matrix $C$ if and only if
    \begin{enumerate}[label=(P\arabic*)]
    \item \label{cond-p1} $P> 0$ (resp. $P\geq 0$),
    \item \label{cond-p2} $P\mathbf 1 = \mathbf 1$, that is $P$ is row-stochastic, and
    \item \label{cond-p3} $\hat P := P\oslash P^T$ is transitive.
\end{enumerate}
\end{proposition}

\begin{proof}
    Let us assume first that $P$ is obtained by row-wise normalization of a positive, symmetric matrix $C$. Then properties (P1) and (P2) are trivially true by definition. In order to see that (P3) holds, let $m:= C\mathbf 1$ and $D:= \mbox{diag}(m)^{-1}$, so that $P= DC$. Since $C$ is symmetric, we get
$$ D^{-1}P = C = C^T = P^T D^{-1}, $$
and thus $m(x)P(x,y)=m(y)P(y,x)$ for all $x,y\in X$, because
$$ D^{-1}(x,x)P(x,y) = \sum_{z\in X} D^{-1}(x,z)P(z,y) = \sum_{z\in X} P^T(x,z) D^{-1}(z,y) = P^T(x,y) D^{-1}(y,y).$$
In other words, it holds that
\begin{equation}\label{PoslashPT-transitive}
    (P\oslash P^T)(x,y) = \frac{m(y)}{m(x)}
\end{equation} 
which is immediately seen to be equivalent to the transitivity of $P\oslash P^T$ by \cite[Proposition 1]{FRS99} and the fact that $P$ is row-stochastic.

Conversely, let the properties (P1)-(P3) hold. Then as observed there must exist a vector $m$ such that \eqref{PoslashPT-transitive} holds, and we can define as above $D:= \mbox{diag}(m)^{-1}$ and $C := D^{-1}P$. Since $P$ is positive, the matrix $P\oslash P^T$ is itself positive, and we can choose $m$ with positive entries. It follows that $C>0$. Given that $P\mathbf 1 = \mathbf 1$, it is immediately seen that 
$$ C\mathbf 1 = D^{-1}P\mathbf 1 = D^{-1}\mathbf 1 = m, $$
so that indeed $P$ is obtained by row-wise normalization of a positive matrix $C$. Finally, $C$ is symmetric because, by the diagonality of $D$,
$$ C(x,y) = m(x)P(x,y) = m(x) (P\oslash P^T)(x,y) P^T(x,y) = m(y)P(y,x)=C(y,x). $$
\end{proof}

Thus a transition matrix $P$ corresponds to an interaction matrix $C$ if and only if conditions (P1)-(P3) hold. If $P,\tilde P$ both verify the random walk data in the sense of Theorem \ref{characterization-new}, and conditions (P1)-(P3) hold for both, then the corresponding interaction matrices $C,\tilde C$ are related as follows:

\begin{proposition}\label{C-from-P-new}

Let $P,\tilde P \in \R^{(N+M)\times (N+M)}$ verify conditions (P1)-(P3) from Proposition \ref{lem:3-cond}, and assume that $P\in GL_M\cdot\tilde P$. Let $C, \tilde C \in \R^{(N+M)\times (N+M)}$ be interaction matrices associated to $P, \tilde P$  respectively. Then there exists an invertible matrix $A \in \R^{M\times M}$ such that $$C =(Id_N\oplus D_{22}^{-1}AD_{22}) \mbox{\emph{diag}}(m\oslash \tilde m)\tilde C (Id_N\oplus A^{-1}).$$ 
    If $A$ is diagonal, then 
    $$C =(Id_N\oplus A) \mbox{\emph{diag}}(m\oslash \tilde m)\tilde C (Id_N\oplus A^{-1}).$$
    If in particular $A=Id_M$, then there exists a scalar $\lambda>0$ such that $C=\lambda\tilde C$.
\end{proposition}

\begin{proof} For general transition matrices $P, \tilde P$ such that $P\in GL_M\cdot \tilde P$ we have
\begin{align*}
    C & = D^{-1}P = D^{-1}(Id_N\oplus A) \tilde P (Id_N\oplus A^{-1})
    \\ & = D^{-1}(Id_N\oplus A)\tilde D\tilde C (Id_N\oplus A^{-1})
    \\ & = D^{-1}(Id_N\oplus A)D\,\mbox{diag}(m\oslash \tilde m)\tilde C (Id_N\oplus A^{-1})
    \\ & = (Id_N\oplus D_{22}^{-1}AD_{22}) \mbox{diag}(m\oslash \tilde m) \tilde C (Id_N\oplus A^{-1}).
\end{align*}

If in particular $A$ is diagonal, then $D_{22}$ and $A$ commute, and it follows
$$C =(Id_N\oplus A) \mbox{diag}(m\oslash \tilde m)\tilde C (Id_N\oplus A^{-1}).$$
This implies that there exist two functions $g_1, g_2:X\rightarrow (0,+\infty)$ such that for all $x,y\in X$ 
$$ C(x,y) = g_1(x)\tilde C(x,y)g_2(y). $$ 

Fix now $x_0\in X$ and assume $A=Id_M$, which implies $P= \tilde P$. By equation \eqref{PoslashPT-transitive} we have $$(m\oslash \tilde m)(x) =  \frac{m(x_0)}{\tilde m(x)}(P\oslash P^T)(x_0,x)  = \frac{m(x_0)}{\tilde m(x)}(\tilde P\oslash \tilde P^T)(x_0,x) = (m\oslash \tilde m)(x_0)=:\lambda$$
for all $x\in X$, and thus 
$$ C = (Id_N\oplus A)\mbox{diag}(m\oslash \tilde m) \tilde C (Id_N\oplus A^{-1}) = \lambda \tilde C. $$ 
\end{proof}

As a result of the previous Proposition, we see that even in the case that the transition matrix $P$ verifying conditions (P1)-(P3) is completely known, the underlying interaction matrix $C$ can only be recovered up to a positive factor $\lambda$. This is the maximum amount of information that we can recover from a transition probability matrix: because the transition matrices $C$ and $\lambda C$, where $\lambda >0$, correspond to the same transition probability matrix, the exact values of the conductivity can never be recovered from the random walk data. In the next section we will consider the recovery of the edge set $E$ and the conductivity $\gamma$ from $C$. The rest of this section is dedicated to the study of the set $\mathcal A \subseteq GL_M$ consisting of those matrices $A$ such that $P=g(A,\tilde P)$ verifies conditions (P1)-(P3), and thus indeed corresponds to an interaction matrix $C$.\\

Let $\tilde P$ be the true transition matrix determined by the unknown graph and conductivity. Since we start from the assumption that $P$ verifies the random walk data, we know that $P\in GL_M\cdot\tilde P$. Moreover, $\tilde P$ verifies conditions (P1)-(P3) by assumption. Thus $\mathcal A$ consists of those invertible matrices which preserve conditions (P1)-(P3). We have the following:

\begin{proposition}\label{prop:char-P2}
    Assume that $\tilde P$ verifies (P1)-(P3) and $P= g(A,\tilde P)$, where $A\in GL_M$. Then $P$ verifies (P2) if and only if $A$ is row-stochastic.
\end{proposition}

\begin{proof}
    Observe that we always have
    $$(Id \oplus A^{-1})P\mathbf 1 = \tilde P(Id\oplus A^{-1})\mathbf 1 = \tilde P v,$$
    where we let $v:= (Id\oplus A^{-1})\mathbf 1 = (\mathbf 1_N, A^{-1}\mathbf 1_M)$. Because $A\in GL_M$, we see that $P\mathbf 1 = \mathbf 1$ is equivalent to $\tilde P v = v$, which in turn is equivalent to $v\in\mathcal E_1$, the eigenspace of $\tilde P$ corresponding to the eigenvalue $1$. Observe that the fractional random walk we study is equivalent to a (regular) random walk in the complete undirected graph with vertex set $\widetilde X$, in which the weight associated to each edge $\{x,y\}$ is $\widetilde C(x,y)$. Therefore, $\tilde P$ is irreducible and non-negative. By the Perron-Frobenius theorem for irreducible non-negative matrices \cite[Theorem 8.4.4]{HJ12}, we deduce that it must be $\mathcal E_1 = \mbox{span}(\mathbf 1)$. This implies that $v=\mathbf 1$, and so $A^{-1}\mathbf 1_M = \mathbf 1_M$. Since $A$ is invertible, this gives the equivalent condition $A \mathbf 1_{M} = \mathbf 1_{M}$, i.e. $A$ is row-stochastic.
\end{proof}

As for condition (P1), characterizing the matrices $P$ which are equivalent to a positive matrix $\tilde P$ (i.e. such that there exists an invertible matrix $M$ such that $P=M\tilde PM^{-1}$) is a surprisingly hard open problem in linear algebra, which we do not explore in this article. 

The structure of the gauge can be summarized in the following way (see Figure \ref{figure-gauge}). 

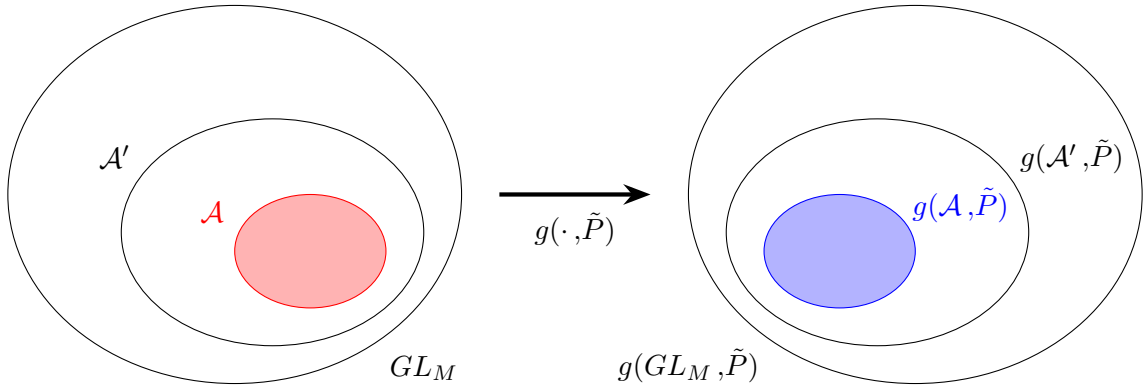
\begin{figure}[!h] 
\centering
\begin{tikzpicture}
  \draw (-0.5,0.5) ellipse (3 and 2.5);
  \draw (0,0) ellipse (2 and 1.5);
  \filldraw[fill=red, draw=red, fill opacity=0.3]  (0.5,-0.25) ellipse (1 and 0.75);
  
  \draw (8.5,0.5) ellipse (3 and 2.5);
  \draw (8,0) ellipse (2 and 1.5);
  \filldraw[fill=blue, draw=blue, fill opacity=0.3] (7.5,-0.25) ellipse (1 and 0.75);

\draw [ultra thick, black, arrows = {-Stealth}] (3,0.5) -- (5,0.5);
\coordinate [label=below:$g(\cdot\,\mbox{,}\tilde P)$]  (A) at (4,0.4);
\coordinate [label=below:$GL_M$]  (B) at (2,-1.5);
\coordinate [label=below:$g(GL_M\,\mbox{,}\tilde P)$]  (B) at (5.5,-1.4);
\coordinate [label=left:$\mathcal A'$]  (C) at (-1.75,1);
\coordinate [label=right:$g(\mathcal A'\,\mbox{,}\tilde P)$]  (D) at (9.75,1);
\coordinate [label=above:\textcolor{red}{$\mathcal A$}]  (E) at (-0.8,0);
\coordinate [label=above:\textcolor{blue}{$g(\mathcal A\,\mbox{,}\tilde P)$}]  (F) at (9.1,0);
\end{tikzpicture}
\caption{The structure of the gauge, as seen in the set $GL_M$ of invertible matrices and in its image under the group action $g$.}
\label{figure-gauge}
\end{figure}
On the left-hand side we see the whole group $GL_M$, its subset $\mathcal A':=\{A\in GL_M : A\mathbf 1_M = \mathbf 1_M\}$, and the red set $\mathcal A$ consisting of those invertible matrices which preserve conditions (P1)-(P3). On the right-hand side we see the images under the function $g(\cdot,\tilde P)$ of the sets on the left-hand side. Recall that the function $g(\cdot,\tilde P)$ is bijective by Lemma \ref{lem:unique}. In particular:
\begin{itemize}
    \item $g(GL_M,\tilde P)$, which we have indicated also by $GL_M\cdot\tilde P$, is the orbit of $\tilde P$ under the group action $g$. By the characterization Theorem \ref{characterization-new}, these are all the matrices which satisfy the random walk data.
    \item By Proposition \ref{prop:char-P2}, we have that $g(\mathcal A',\tilde P)$ consists of all the matrices which satisfy both the random walk data and (P2).
    \item By Proposition \ref{lem:3-cond}, to say that a (transition) matrix $P$ satisfies conditions (P1)-(P3) is equivalent to say that it is obtained by row-wise normalization of a positive, symmetric matrix $C$. Thus the elements of $g(\mathcal A,\tilde P)$ are exactly the transition matrices associated to a positive, symmetric matrix $C$ which satisfy the random walk data.
\end{itemize}

\section{Recovery of the edge set and conductivity} 

 We now deal with the final recovery of the edge set $E$ and the conductivity $\gamma$, given the interaction matrix $C$. We will observe that the edge set can be completely recovered in some cases, while the conductivity can only be determined up to a positive factor. \\

The following is the main result of this section:

\begin{theorem}[Reconstruction]\label{main-result}
Let $G=(X,E)$ be a graph satisfying \ref{adm-cond-2}, and let $\sigma_1,\sigma_2:X\rightarrow (0,+\infty)$. Assume that the  matrix $f(x,y):= \frac{\sigma_1(x)\sigma_2(y)}{ d_G(x,y) }$ is known for all $x\neq y\in X$. Then the edge set $E$ can be uniquely recovered, and $\sigma_1, \sigma_2$ can be recovered up to a positive factor.
\end{theorem}

\begin{remark}\label{significance}
The significance of Theorem \ref{main-result} is that it allows us to recover the edge set relative to a fractional conductivity random walk in the assumption that $g_1(x) \tilde C(x,y) g_2(y)$ is known everywhere besides on the diagonal. In fact, we can define
$$ f(x,y)=\left(g_1(x) \tilde C(x,y) g_2(y)\right)^{1/(n+2s)}, \qquad \sigma_j(x) = \left(g_j(x)\gamma^{1/2}(x)\right)^{1/(n+2s)} \quad \mbox{ for } j=1,2,$$
which allows us to deduce $E$ and $g_j\gamma^{1/2}$, $j=1,2$ (up to a positive factor) by the theorem. If in particular $g_1(x)=g_2(x)=\lambda^{1/2}$ for all $x\in X$, then the conductivity can itself be recovered up to a positive factor.
\end{remark}

\begin{proof}
\textbf{Step 1.} Fix three distinct points $a,b,c \in X$, which is allowed by the assumption that $G$ satisfies condition \ref{adm-cond-2}, and define the function $F_{abc}:X\setminus \{a,b\}\rightarrow \R$ by
$$ F_{abc}(x) := \frac{f(a,x)f(b,c)}{f(b,x)f(a,c)} = \frac{d_G(b,x)d_G(a,c)}{d_G(a,x)d_G(b,c)} = q_{abc} \frac{d_G(b,x)}{d_G(a,x)}, $$
where $q_{abc}\in \Q$ is some unknown but fixed rational number. Observe that $F_{abc}$ can be computed from our knowledge of $f$, and thus we can determine two points $\underline x, \overline x\in X\setminus\{a,b\}$ realizing the minimum and maximum of $F_{abc}$. In general, neither the minimum nor the maximum point need to be unique, but it does not affect our arguments. In particular, we can compute the ratio 
$$ R_{ab}:=\frac{\max F_{abc}}{\min F_{abc}}=\frac{F_{abc}(\overline x)}{F_{abc}(\underline x)} = \frac{d_G(b,\overline x)d_G(a,\underline x)}{d_G(a,\overline x)d_G(b,\underline x)}, $$
which is independent of $q_{abc}$ and, therefore, of the point $c\in X$. 
\\

\textbf{Step 2 (localization of the extremal points).} Fix two distinct points $a,b\in X$. For any $x\in X\setminus\{a,b\}$ we can find two paths of minimum length between $x$ and $a,b$ respectively. 
Let $x_1$ be the second point along the fixed path from $x$ to $b$. Given the minimality of the length of such path, it must necessarily be $$d_G(b,x_1)=d_G(b,x)-1.$$ On the other hand, due to the minimality of the chosen path between $a$ and $x$, the quantity $d_G(a,x_1)$ is either equal, one more or one less than $d_G(a,x)$ (see Figure \ref{figure-one-more-less}).

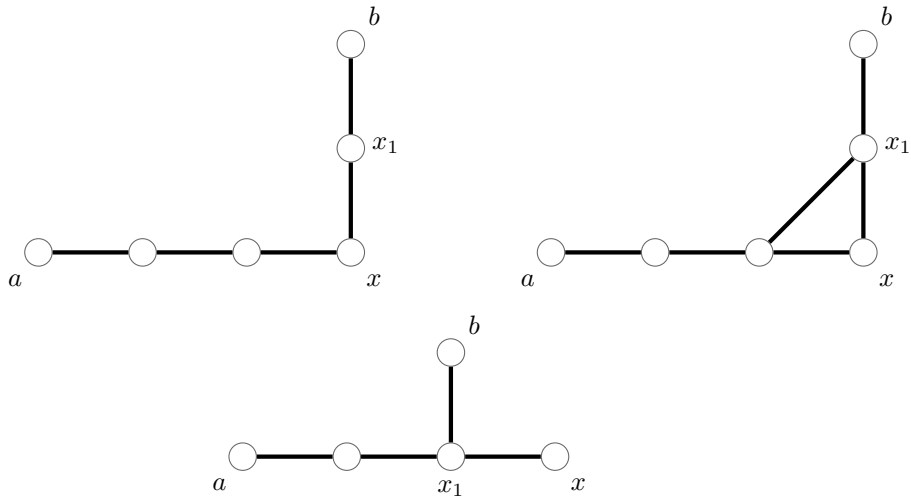
\begin{figure}[!h]
\centering
\begin{tikzpicture}[
Jnode/.style={circle, draw=blue!60, fill=blue!5, thin, minimum size=1mm},
J0node/.style={circle, draw=black!60, thin, minimum size=1mm},
Enode/.style={}
]
\node[J0node][label={[xshift=-.8em, yshift=-2em]\small $a$}] (N1) {};
\node[J0node] (N2) [right=of N1] {};
\node[J0node] (N3) [right=of N2] {};
\node[J0node][label={[xshift=.8em, yshift=-2em]\small $x$}] (N4) [right=of N3] {};
\node[J0node][label={[xshift=1.2em, yshift=-1em]\small $x_1$}] (N5) [above=of N4] {};
\node[J0node][label={[xshift=.8em, yshift=-.2em]\small $b$}] (N6) [above=of N5] {};
\node[Enode] (N7) [right=of N4] {};
\node[J0node][label={[xshift=-.8em, yshift=-2em]\small $a$}] (N8) [right=of N7] {};
\node[J0node] (N9) [right=of N8] {};
\node[J0node] (N10) [right=of N9] {};
\node[J0node][label={[xshift=.8em, yshift=-2em]\small $x$}] (N11) [right=of N10] {};
\node[J0node][label={[xshift=1.2em, yshift=-1em]\small $x_1$}] (N12) [above=of N11] {};
\node[J0node][label={[xshift=.8em, yshift=-.2em]\small $b$}] (N13) [above=of N12] {};
\node[J0node][label={[xshift=.8em, yshift=-.2em]\small $b$}] (N14) [below=of N7] {};
\node[J0node][label={[xshift=0em, yshift=-2.2em]\small $x_1$}] (N15) [below=of N14] {};
\node[J0node][label={[xshift=.8em, yshift=-2em]\small $x$}] (N16) [right=of N15] {};
\node[J0node] (N17) [left=of N15] {};
\node[J0node][label={[xshift=-.8em, yshift=-2em]\small $a$}] (N18) [left=of N17] {};
\draw[ultra thick] (N10.north east) -- (N12.south west);
\draw[ultra thick] (N1.east) -- (N2.west);
\draw[ultra thick] (N2.east) -- (N3.west);
\draw[ultra thick] (N3.east) -- (N4.west);
\draw[ultra thick] (N8.east) -- (N9.west);
\draw[ultra thick] (N9.east) -- (N10.west);
\draw[ultra thick] (N10.east) -- (N11.west);
\draw[ultra thick] (N15.east) -- (N16.west);
\draw[ultra thick] (N18.east) -- (N17.west);
\draw[ultra thick] (N17.east) -- (N15.west);
\draw[ultra thick] (N4.north) -- (N5.south);
\draw[ultra thick] (N5.north) -- (N6.south);
\draw[ultra thick] (N11.north) -- (N12.south);
\draw[ultra thick] (N12.north) -- (N13.south);
\draw[ultra thick] (N15.north) -- (N14.south);
\end{tikzpicture}
\caption{The quantity $d_G(a,x_1)$ is either one more, the same, or one less than $d_G(a,x)$. Due to the minimality of the path between $x$ and $a$, no other choices are possible.}
\label{figure-one-more-less}
\end{figure}

In any case, after some easy calculations, if the relation $d_G(b,x)\leq d_G(a,x)$ holds we have that
$$ \frac{d_G(b,x)}{d_G(a,x)} \geq \frac{d_G(b,x_1)}{d_G(a,x_1)}, \qquad \mbox{and } \qquad d_G(b,x_1)\leq d_G(a,x_1). $$
Repeating this process a finite amount of times, we find that $$ \frac{d_G(b,x)}{d_G(a,x)} \geq \frac{d_G(b,x')}{d_G(a,x')}, $$
where $x'$ is the second point in the fixed path from $b$ to $x$, and as such verifies $d_G(b,x')=1$. Therefore, if there exist points $x\sim b$ with $x\neq a$, it is clear that $\underline x$ must be among them. Similarly, if there exist points $x\sim a$ with $x\neq b$, necessarily $\overline x$ is one of them. Observe that if both conditions are verified, then $R_{ab}\in \N$.
\\

Assume instead that there is no $x\sim b$ such that $x\neq a$, that is, $b\in L$ is a leaf and its unique neighbour is $a$. In this case for every $x\in X\setminus\{a,b\}$ we have $\frac{d_G(b,x)}{d_G(a,x)} = \frac{M+1}{M}$
for some $M\geq 1$, and thus $R_{ab}= \frac{2M^*}{M^*+1}$, where $M^*$ is the maximum allowed $M$ (see Figure \ref{figure-M-star}). If we had $M^*=1$, then $G$ would be a star centered at $a$, which is not allowed by the admissibility assumption for $G$. Therefore $M^*>1$ and $R_{ab} \in (1,2)$, which in particular implies $R_{ab}\notin\N$. A similar result is obtained in the assumption that there are no points $x\sim a$ with $x\neq b$. Since $|X|>2$, the two conditions from the previous paragraph can not fail together.\\

\begin{figure}[!h]
\centering
\begin{tikzpicture}[
J0node/.style={circle, draw=black!60, thin, minimum size=1mm}
]
\node[J0node][label={[xshift=-.8em, yshift=-2em]\small $b$}] (N1) {};
\node[J0node][label={[xshift=-.8em, yshift=-2em]\small $a$}] (N2) [right=of N1] {};
\node[J0node][label={[xshift=-.8em, yshift=-2em]\small $\overline x$}] (N3) [right=of N2] {};
\node[J0node] (N4) [right=of N3] {};
\node[J0node] (N5) [above=of N4] {};
\node[J0node] (N6) [below=of N4] {};
\node[J0node][label={[xshift=-.8em, yshift=-.2em]\small $\underline x$}] (N7) [right=of N6] {};
\draw[ultra thick] (N6.north west) -- (N3.south east);
\draw[ultra thick] (N5.south west) -- (N3.north east);
\draw[ultra thick] (N1.east) -- (N2.west);
\draw[ultra thick] (N2.east) -- (N3.west);
\draw[ultra thick] (N3.east) -- (N4.west);
\draw[ultra thick] (N6.east) -- (N7.west);
\draw[ultra thick] (N4.south) -- (N6.north);
\end{tikzpicture}
\caption{If $b$ is a leaf and $a$ is its unique neighbour, then $R_{ab}=\frac{2M^*}{M^*+1}$, where $M^*$ is the eccentricity of $a$. In this case, we have that $M^*=d_G(a,\underline x)$.}
\label{figure-M-star}
\end{figure}
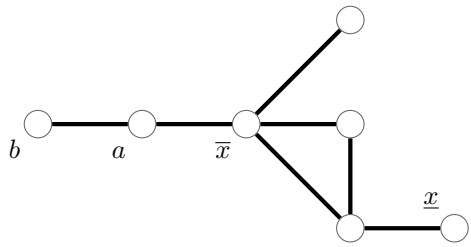

\textbf{Step 3 (recovery of leaf set $L$).} Let $a,b\in X$ be distinct points. If $R_{ab}\notin\N$, then we know that exactly one of $a,b$ is a leaf, and $a\sim b$.
Thus the set $Y:=\{ \{a,b\} \in X^2 : R_{ab}\notin\N\}$ is comprised of all the pairs constituted by a leaf and its unique neighbour. However, given that the couples constituting $Y$ are not ordered, it remains to be determined which element of each couple is the leaf, and which is the neighbouring point. If $a,b\in X$ are two distinct leaves, we shall necessarily have that $\overline x$ and $\underline x$ are their respective unique neighbours, which implies that $R_{ab} = (d_G(a,b)-1)^2$ is a perfect square. The same happens when $a,b\in X$ are two distinct neighbours of leaves, because in this case $\overline x, \underline x$ are the corresponding leaves and thus $R_{ab}=(d_G(a,b)+1)^2$ is again a perfect square. However, if $a,b\in X$ are a leaf and the neighbour of a different leaf, we obtain $R_{ab} = (d_G(a,b)-1)(d_G(a,b)+1)= d_G(a,b)^2-1$, which is never a perfect square. These considerations allow us to split the set $\bigcup Y$ into two subsets $Y_1, Y_2$, one containing all the leaves and the other containing all their neighbours, in such a way that each $\{a,b\}\in Y$ has an element from each set. However, while the partition of $\bigcup Y$ is itself known, we do not yet know which of the two sets $Y_1, Y_2$ is the set $L$ of the leaves, and which in turn is the set $N$ of the neighbours of leaves. (see Figure \ref{figure-leaves-neighbours}). 

\begin{figure}[!h]
\centering
\begin{tikzpicture}[
Jnode/.style={circle, draw=blue!60, fill=blue!5, thin, minimum size=1mm},
J0node/.style={circle, draw=black!60, thin, minimum size=1mm}
]
\node[J0node] (N1) {};
\node[J0node] (N2) [below=of N1] {};
\node[J0node] (N3) [right=of N1] {};
\node[J0node] (N4) [below=of N3] {};

\node[Jnode][label={[xshift=-.8em, yshift=-.2em]\small \color{blue} $N_1$}] (N5) [left=of N1]  {};
\node[Jnode][label={[xshift=-.8em, yshift=-.2em]\small \color{blue} $L_1$}] (N6) [left=of N5]  {};
\node[Jnode][label={[xshift=-.8em, yshift=-.2em]\small \color{blue} $N_2$}] (N8) [right=of N3] {};
\node[Jnode][label={[xshift=-.8em, yshift=-.2em]\small \color{blue} $L_2$}] (U1) [right=of N8] {};
\node[J0node] (U2) [below=of N8] {};
\node[J0node] (U4) [below=of N5] {};

\draw[ultra thick] (N1.south) -- (N2.north);
\draw[ultra thick] (N1.east) -- (N3.west);
\draw[ultra thick] (N2.east) -- (N4.west);
\draw[ultra thick] (N2.north east) -- (N3.south west);
\draw[ultra thick] (N3.south east) -- (U2.north west);

\draw[ultra thick, blue] (N6.east) -- (N5.west);
\draw[ultra thick, blue] (N8.east) -- (U1.west);
\draw[ultra thick] (N5.east) -- (N1.west);
\draw[ultra thick] (N5.south) -- (U4.north);
\draw[ultra thick] (N3.east) -- (N8.west);
\draw[ultra thick] (U4.east) -- (N2.west);
\draw[ultra thick] (N4.east) -- (U2.west);
\end{tikzpicture}
\caption{An example of graph with two leaf-neighbour pairs highlighted. Here $\{L_1, N_1\}, \{L_2,N_2\} \in Y$. Observe that the couples are not ordered. }
\label{figure-leaves-neighbours}
\end{figure}
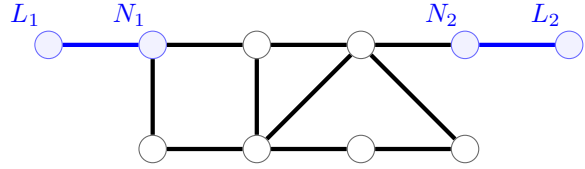

Assume $\{a,b\}\in Y$ with $R_{ab}> 4/3$, and fix any $c\in X\setminus\{a,b\}$. Let $\underline x \in \arg\min F_{abc}$, and consider the set $$Z:=\bigcup_{y\in \arg\max F_{abc}} \arg\max F_{y\underline xa}.$$
If $a\in N$, then $\arg\max F_{abc} = \{y\in X: y\neq b, y\sim a\}$. Thus the minimum path between $a$ and $\underline x$ must pass through some $y^*\in \arg\max F_{abc}$, because $b\in L$. This implies that $d_G(y^*,\underline x)= d_G(a,\underline x)-1$, and the path from $a$ to $\underline x$ with $a$ removed is a minimum path between $y^*$ and $\underline x$. Notice that it can not be the case that $y^*=\underline x$, because otherwise there would not be points in $X$ of distance larger than $1$ from $a$, and therefore $G$ would be a star centered at $a$. Observe that $\arg\max F_{y^*\underline xa}$ consists of all $z\sim y^*$ with maximum distance from $\underline x$. Given that by the triangle inequality $d_G(z,\underline x) \leq d_G(y^*,\underline x)+1 = d_G(a,\underline x)$, and $a\sim y^*$, we conclude that $a\in \arg\max F_{y^*\underline xa}$, and therefore to $Z$. Thus $\{a,b\}\cap Z \neq \emptyset$ (see Figure \ref{figure-step3-1}).

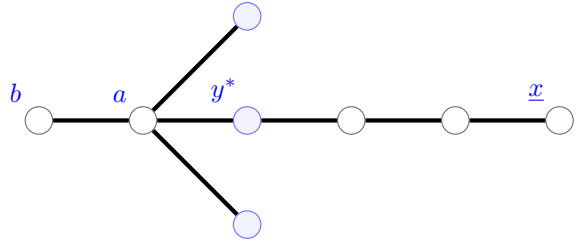
\begin{figure}[!h]
\centering
\begin{tikzpicture}[
Jnode/.style={circle, draw=blue!60, fill=blue!5, thin, minimum size=1mm},
J0node/.style={circle, draw=black!60, thin, minimum size=1mm}
]
\node[J0node] (N1) {};
\node[J0node] (N3) [right=of N1] {};

\node[Jnode][label={[xshift=-.8em, yshift=-.2em]\small \color{blue} $y^*$}] (N5) [left=of N1]  {};
\node[J0node][label={[xshift=-.8em, yshift=-.2em]\small \color{blue} $a$}] (N6) [left=of N5]  {};
\node[J0node][label={[xshift=-.8em, yshift=-.2em]\small \color{blue} $b$}] (U6) [left=of N6]  {};
\node[Jnode] (N7) [above=of N5] {};
\node[J0node][label={[xshift=-.8em, yshift=-.2em]\small \color{blue} $\underline x$}] (N8) [right=of N3] {};
\node[Jnode] (U4) [below=of N5] {};

\draw[ultra thick] (N1.east) -- (N3.west);

\draw[ultra thick] (N6.east) -- (N5.west);
\draw[ultra thick] (N6.north east) -- (N7.south west);
\draw[ultra thick] (N6.south east) -- (U4.north west);
\draw[ultra thick] (U6.east) -- (N6.west);
\draw[ultra thick] (N5.east) -- (N1.west);
\draw[ultra thick] (N3.east) -- (N8.west);
\end{tikzpicture}
\caption{The situation in the second part of Step 3, when $b\in L$. We have indicated the set $\arg\max F_{abc}$ in blue.  }
\label{figure-step3-1}
\end{figure}

\noindent If instead $a\in L$, we have $b\sim \underline x$ and $d_G(b,y) = M^* > 2$ for all $y\in \arg\max F_{abc}$, by the assumption $R_{ab}>4/3$. Now $d_G(y,\underline x) \geq d_G(y,b)-1 >1$, which means that $y\not\sim \underline x$. Thus necessarily if $z\in\arg\max F_{y\underline xa}$ it must be $d_G(y,z)=1$ by Step 2, and in particular $z\notin \{a,b\}$. Therefore, $\{a,b\}\cap Z=\emptyset$ (see Figure \ref{figure-step3-2}).

\begin{figure}[!h]
\centering
\begin{tikzpicture}[
Jnode/.style={circle, draw=blue!60, fill=blue!5, thin, minimum size=1mm},
J0node/.style={circle, draw=black!60, thin, minimum size=1mm}
]
\node[J0node][label={[xshift=-.8em, yshift=-.2em]\small $b$}] (N1) {};
\node[J0node][label={[xshift=-.8em, yshift=-.2em]\small $\underline x$}] (N2) [below=of N1] {};
\node[J0node] (N3) [right=of N1] {};
\node[J0node][label={[xshift=-.8em, yshift=-.2em]\small  $a$}] (N5) [left=of N1]  {};
\node[Jnode][label={[xshift=-.8em, yshift=-.2em]\small \color{blue} $z$}] (N8) [right=of N3] {};
\node[Jnode] (N9) [above=of N8] {};
\node[J0node][label={[xshift=.8em, yshift=-.2em]\small $y$}] (U1) [right=of N8] {};
\node[Jnode] (U2) [below=of N8] {};

\draw[ultra thick] (N1.south) -- (N2.north);
\draw[ultra thick] (N1.east) -- (N3.west);

\draw[ultra thick] (N8.east) -- (U1.west);
\draw[ultra thick] (N5.east) -- (N1.west);
\draw[ultra thick] (N3.east) -- (N8.west);
\draw[ultra thick] (N9.south east) -- (U1.north west);
\draw[ultra thick] (U2.north east) -- (U1.south west);
\end{tikzpicture}
\caption{The situation in the second part of Step 3, when $a\in L$. We have indicated the set $\arg\max F_{y\underline x a}$ in blue. }
\label{figure-step3-2}
\end{figure}

\noindent By the above argument we are able to distinguish in each pair $\{a,b\}\in Y$ the leaf from its neighbour, in the assumption that $R_{ab}>4/3$. Observe that it suffices to correctly identify just one pair belonging to $Y$ in order to completely determine $L$ and $N$. Thus assuming that $\max_{\ell \in L} e(\ell) > 3$ allows us to recover $L$. 
\\

\textbf{Step 4 (final recovery of $E$ and $\sigma$).} If $a,b\in L$ are two distinct leaves, then by the above step we know that $R_{ab} = (d_G(a,b)-1)^2$, which lets us compute the distance $d_G(a,b)$. Thus $d_G$ can be recovered on $L\times L$, which of course entails that $d_G$ can also be recovered on $L\times N$ and $N\times N$ by the properties of the set $Y$. Let now $a,c\in L$, $b\in N$ such that $b\sim a$, and $x\in X$. We have
$$F_{abc}(x)\frac{d_G(b,c)}{d_G(a,c)} = \frac{d_G(b,x)}{d_G(a,x)} = \frac{d_G(a,x)-1}{d_G(a,x)}, $$
because any minimum path between $a$ and $x$ must pass through $b$. Since the left-hand side is completely known, we can compute $d_G(a,x)$. This means that $d_G$ can be recovered on $L\times X$. Finally, let $x,y\in X$ and $a,c\in L$. We can compute
$$d_G(y,x)=F_{ayc}(x)\frac{d_G(y,c)d_G(a,x)}{d_G(a,c)}, $$
which completes the recovery of $d_G$ on $X\times X$. Thus we have recovered the edge set $E$.
\\

The knowledge of $E$ and $f$ is enough to deduce the off-diagonal entries of the matrix $\Sigma := \sigma_1^T\sigma_2 \in \R^{(N+M)\times (N+M)}$. If now $x,y,z\in X$, we can write
$$ \sigma_1(x)\sigma_2(x) = \frac{\sigma_1(y)\sigma_2(x) \sigma_1(x)\sigma_2(z)}{\sigma_1(y)\sigma_2(z)}= \frac{f(y,x)d_G(x,y) f(x,z)d_G(x,z)}{f(y,z)d_G(y,z)}, $$
and thus we can compute the diagonal entries of $\Sigma$, given that the right-hand side is completely known. Since $\Sigma$ does not vanish, and it has a factorization in matrices of rank $1$, we deduce rank$(\Sigma)=1$. Thus $\Sigma = \sigma_1^T\sigma_2$ is a full-rank factorization, which means that $\sigma_1, \sigma_2$ can be computed up to a positive factor by \cite[Theorem 2]{PO99}. 
\end{proof}

The discussion of Section \ref{sec:gauge}, Theorem \ref{main-result} and Remark \ref{significance} now give the following corollary:

\begin{corollary}\label{cor:use-main-theorem}
    If $P=\tilde P$, then $E=\tilde E$ and the conductivities $\gamma$ and $\tilde \gamma$ differ by a positive factor.
\end{corollary}

\section{Proof of the main result}

We are now in the position of proving the main Theorem \ref{th:main}.

\begin{proof}[Proof of Theorem \ref{th:main}]
    The proof follows as a combination of the results obtained in the previous sections, and is therefore already contained in the above discussions. We shall however summarize it here for the convenience of the reader. 
    
    Let $P,\tilde P$ be transition matrices, and assume that the corresponding random walk data verifies
 $$ \Lambda_3(P)=\Lambda_3(\tilde P). $$   
    By the admissibility assumption for $G$ and Theorem \ref{characterization-new}, this is already enough in order to determine the cardinality of $X$ and deduce the fact that $P$ belongs to the orbit $GL_M\cdot\tilde P$. Moreover, we have proved that any further knowledge of the transition probabilities between observable points for larger values of $k\in\N$, that is the assumption $$ \Lambda_\omega(P)=\Lambda_\omega(\tilde P), $$ would not help any further in the determination of $P$. However, not all the elements of the orbit $GL_M\cdot\tilde P$ are transition matrices, as discussed in Section \ref{sec:gauge}: the matrix $P$ rather belongs to $g(\mathcal A, \tilde P)$, where the set $\mathcal A$ comprises all those matrices $A\in GL_M$ such that $g(A,\tilde P)$ is obtained by row-wise normalization of a positive, symmetric matrix. By Proposition \ref{prop:char-P2}, all matrices in $\mathcal A$ must be row-stochastic. \\ 
    If in particular $P=\tilde P$, by Theorem \ref{main-result} and Corollary \ref{cor:use-main-theorem} we have $E=\tilde E$, and the conductivities $\gamma$ and $\tilde \gamma$ differ by a positive factor. Moreover, both $\tilde E$ and $\tilde \gamma$ (up to a positive factor) can be reconstructed from the random walk data.
\end{proof}

\section{Open problems and conjectures}\label{sec:open-problems}

We dedicate the last section of the article to a brief discussion of some open problems which arose during the preparation of the paper. These all seem quite interesting developments of our initial study, and we plan to come back to them in future works. We are grateful to the many colleagues whose questions inspired the following list. \\

\begin{itemize}
    \item \textbf{Stability for the random walk data.} Throughout the paper, we have used random walk data in the form $\Lambda_\omega(P)$ (or rather $\Lambda_K(P)$), which amounts to knowing the transition probabilities in any number of jumps between observable points. As observed in Section 3, this kind of data can be produced from observation of the random walk $H^{x_0}_t$ on $B$ alone, as the probability of events of the kind $\{H^{x_0}_t=x, H^{x_0}_{t+k}=y\}$, for $x,y\in B$ and $k\in\N$. However, this procedure requires in principle the observation of the \emph{whole} random walk, and in an experimental setting the precision of the measured data will depend on the length of the observation period. It is therefore natural to wonder in what measure a small error introduced in the measured transition probabilities will interfere with the reconstruction of the edge set $E$ and the conductivity $\gamma$. Given the rigid structure imposed by the constraint of working on a finite graph, it seems reasonable to expect that one could be presented with two kids of situations. For most cases, we anticipate that a small error in the random walk data will only cause a controllable error in the reconstruction of $\gamma$, leaving the reconstruction of the edge set $E$ unchanged; however, we expect there to be borderline situations in which the measurement error will cause a sharp switch between two possibly very different configurations of the graph. We believe this possible phenomenon to be very worth of attention, as to our knowledge it has not been observed before.\\
    \item \textbf{Fractional conductivity on a manifold.} As explained in the introduction, one of the motivations of our article is to serve as a stepping stone for the study of the inverse problem for the fractional conductivity operator on a manifold. As discussed, there are still a number of problems to be solved in this direction, including but not limited to the generalization of the technique presented in the Euclidean setting in \cite{Val09} for the fractional Laplacian and \cite{C20} for the fractional conductivity operator, the construction of suitable graph approximations for manifolds, and the very interesting question of the relation between our discrete random walk data and the inverse problem data on a manifold, which will presumably be given in any of the usual Dirichlet-to-Neumann or Cauchy forms. We shall consider these questions in future works. \\ 
    \item \textbf{Relation with the Unique continuation property.} The characterization Theorem \ref{characterization-new} shows that the random walk data $\Lambda_3(P)$ is both necessary and sufficient for the best possible reconstruction of the transition matrix $P$. It should be noticed that this is a striking nonlocal phenomenon which can not happen in the case of the classical random walk on a graph, where only jumps between adjacent points are allowed. This is due to the fact that in the classical case the random walk data of the form $\Lambda_3(P)$ does not contain information about vertices whose distance from the observable set is larger than $3$ (as these are never reached by the random walker in just $3$ steps), and thus in this situation the full random walk data $\Lambda_\omega(P)$ is needed. On the other hand, the fractional conductivity operator on a manifold obtained by the procedure discussed above will presumably present nonlocal behaviours akin to the Unique continuation property (UCP) for the fractional Laplacian (see \cite{GSU20} and all subsequent works). The study of the relation between the two described nonlocal properties would be of great interest, as this could help shed new light on these fundamental phenomena and perhaps provide an alternative proof of the UCP for fractional operators. \\
    \item \textbf{Possible generalizations.} At the moment of writing, it is not known whether the admissibility conditions \ref{adm-cond-1} and \ref{adm-cond-2}, which we have shown to be sufficient for proving Theorem \ref{th:main}, are also necessary for uniqueness. It would be interesting to extend our results to more general geometrical assumptions, especially in light of the fact that this would require a different method of proof which does not rely on distance functions for leaves.
\end{itemize}

\section{Funding}
This work was was partially supported by a AdG project 101097198 of the European Research Council, Centre of Excellence of Research Council of Finland and the FAME flagship of the Research Council of Finland (grant 359186).

\begin{spacing}{2}  
    \bibliography{refs}
\end{spacing}

\bibliographystyle{alpha}

\end{document}